\documentclass[12pt]{article}

\usepackage{amsmath}
\usepackage{amsfonts}
\usepackage{amsthm}
\usepackage{mathrsfs}
\usepackage{graphicx}
\usepackage{framed}
\usepackage{enumitem}
\usepackage[pagewise, displaymath, mathlines]{lineno}
\usepackage{fancyhdr,lipsum}
\usepackage[colorlinks = true,
            linkcolor = blue,
            urlcolor  = blue,
            citecolor = blue,
            anchorcolor = blue]{hyperref}

\newtheorem{theorem}{Theorem}[section]
\newtheorem{lemma}[theorem]{Lemma}
\newtheorem{corollary}[theorem]{Corollary}
\newtheorem{proposition}[theorem]{Proposition}
\newtheorem{definition}[theorem]{Definition}

\theoremstyle{remark}

\theoremstyle{definition}
\newtheorem{Q}{Question}[section]

\def\<{\langle}
\def\>{\rangle}

\def\dist{{\rm dist}}

\newcommand{\romd}{\mathrm{d}}

\newcommand{\Ical}{\mathcal{I}}
\newcommand{\aaa}{\alpha}
\newcommand{\bb}{\beta}
\newcommand{\cc}{\gamma}
\newcommand{\dd}{\delta}

\begin{document}

\title{On the Assouad dimension of differences of self-similar fractals}

\author{Alexandros Margaris, Eric J. Olson and James C. Robinson}
\date{}

\maketitle

\begin{abstract}
If $X$ is a set with finite Assouad dimension, it is known that the Assouad dimension of $X-X$ does not necessarily obey any non-trivial bound in terms of the Assouad dimension of $X$. In this paper, we consider self-similar sets on the real line and we show that if a particular weak separation condition is satisfied, then the Assouad dimension of the set of differences is bounded above by twice the Assouad dimension of the set itself. We then apply this result to a particular class of asymmetric Cantor sets.
\end{abstract}

\section{Introduction}

We say that a subset $V$ of a metric space $(X,d)$ is $(M,s)$-homogeneous if for every $x \in V$ and $r > \rho > 0$
$$N_{V}(r,\rho) := \sup_{x \in V} N(V \cap B(x,r), \rho) \leq M\left(\frac{R}{r}\right)^{s},$$
where $N(A, \rho)$ denotes the minimum number of balls of radius $\rho$ centred at $A$ required to cover $A$. 

The Assouad dimension of $V \subset (X,d)$, $\romd_{A}(V)$ is defined as the infimum of all $s >0$ such that $V$ is $(M,s)$-homogeneous for some $M >0$.

It is well known (see for example Chapter 9 in the book of Robinson \cite{JCR}) that all subsets of the Euclidean space $\mathbb{R}^{d}$ have finite Assouad dimension, which is bounded above by $d$. 

We also want to recall the definition of the (upper) box--counting dimension which is a lower bound for the Assouad dimension.

\begin{definition}\label{BC2}
Suppose that $(E,d)$ is a metric space. Let $X$ a compact subset of $E$ and let $N(X, \epsilon)$ denote the minimum number of balls of radius $\epsilon$ with centres in $X$ required to cover $X$. The upper box-counting dimension of $X$ is 

\begin{equation}\label{BC3}
\romd_{B}(X) = \limsup_{\epsilon \rightarrow 0} \frac{\log N(X, \epsilon)}{-\log \epsilon}.
\end{equation}

\end{definition}

For the rest of the paper, we will refer to $\romd_{B}(X)$ simply as the box--counting dimension of $X$.

Note that it follows from the definition that if $d > \romd_{B}(X)$, then there exists some positive constant $C=C_{d}$, such that 

\begin{equation}
N(X, \epsilon) \leq C \epsilon^{-d}.
\end{equation}

Olson \& Robinson \cite{OR} showed that if $X$ is a subset of a Hilbert space such that $X-X$ is $(M,s)$-homogeneous, then $X$ admits almost bi--Lipschitz embeddings into Euclidean spaces. Robinson \cite{Rob} extended the result for subsets of Banach spaces. Unfortunately, the fact that $X$ is homogeneous does not necessarily imply that $X-X$ is also homogeneous (examples of such sets can be found in Chapter 9 in the book of Robinson \cite{JCR} or in the PhD thesis of Margaris \cite{Margaris}).

In this paper, we study attractors of Iterated Function Systems. We know from the above discussion that attractors of infinite-dimensional dynamical systems such that the set of differences is $(M,s)$-homogeneous admit embeddings into Euclidean spaces, without losing information from the dynamical system. When the attractor is $(M,s)$-homogeneous, it is a natural question to consider when this property is inherited by the set of differences. 

We focus on self-similar sets, i.e.\ attractors of systems of contracting similarities in the context of a Euclidean space $\mathbb{R}^{s}$. We first want to set up our theory in the general context of any complete metric space $(X,d)$ and then concentrate on systems in $\mathbb{R}^{s}$.
Suppose $(X,d)$ is a complete metric space and let $\mathcal{I} = \{1,...,|\mathcal{I}|\}$ be a finite set of indices. We say that $\mathcal{F} = \{f_{i} \colon X \to X\}_{i \in \mathcal{I}}$ is a system of contracting similarities, if for all $i \in \mathcal{I},$ we have 
$$d(f_{i}(x), f_{i}(y)) = c_{i} \, d(x,y),$$ for some $0< c_{i} < 1.$ 
Then, these maps are obviously contractions, so by the Banach fixed-point theorem they all have fixed points in $X$. We then say that a non-empty compact set $K \subset X$ is an attractor of the system if 
$$K = \bigcup_{i=1}^{|I|} f_{i}(K).$$
It has been proved by Hutchinson \cite{HUTCH} that every system $\mathcal{F}$ in a complete metric space $X$ defines a unique attractor $K$.

We now introduce some notation. Let $c_{min} = \min \{c_{i} : i\in I\}$ and $c_{max} = \max \{c_{i} : i\in I\}$.
Let $\Ical^{*} = \cup_{k \geq 1} \Ical^{k}$ be the set of all finite sequences with entries in $\Ical$. For
$$\alpha = (i_{1},...,i_{k}) \in \Ical^{*},$$ we write
$$f_{\alpha} = f_{i_{1} i_{2} \cdots i_{k}}= f_{i_{1}} \circ f_{i_{2}}\circ...\circ f_{i_{k}} ,$$
and 
$$c_{\aaa} = c_{i_{1}}...c_{i_{k}}.$$
Let also $$\bar{\aaa} = (i_{1},...,i_{k-1})$$ 
We also define for any $b<1$, 
$$I_{b} = \{\aaa \in \Ical^{*} : c_{\aaa} \leq b < c_{\bar{\aaa}}\},$$
the indices of maps at scale $b$.
Finally, we define $C(I)$ to be the set all infinite sequences of integers $(i_{p})_{p=1}^{\infty}$, with entries in $I$. We now state without proof some general properties of the attractor $K$, that we will need in what follows. For the proofs, see the paper of Hutchinson \cite{HUTCH}.

\begin{proposition}\label{properties of K} Suppose that $(\mathcal{F},K)$ is a system of contracting similarities with attractor $K$. Then we have the following
\begin{enumerate}
\item For any given $b<1$, $ K = \bigcup_{\aaa \in I_{b}} f_{\aaa}(K).$
\item $K \supseteq f_{i_{1}}(K) \supseteq f_{i_{1}i_{2}}(K) \supseteq \cdots \supseteq  f_{i_{1}...i_{p}}(K) \supseteq \cdots$ and $\bigcap_{p=1}^{\infty} f_{i_{1}...i_{p}}(K)$ is a singleton, which is denoted by $k_{v}$, for $v = (i_{1},i_{2}, \cdots, i_{p}, \cdots) \in C(I).$ $K$ is the union of all these singletons.
\end{enumerate}
\end{proposition}

Suppose now that $(\mathcal{F},K)$ is a system of contracting similarities in $\mathbb{R}^{s}$ with an attractor $K$. One can show (see Hutchinson \cite{HUTCH}) that a function $f \colon \mathbb{R}^{s} \to \mathbb{R}^{s}$ is a contracting similarity if and only if there exist $0<c_{f}<1$, $q_{f} \in \mathbb{R}^{s}$ such that
\[
	f(x) = c_{f} O_{f}(x) + q_{f},
\]
where $O_{f} \colon \mathbb{R}^{s} \to \mathbb{R}^{s}$ is an orthogonal transformation. 

The computation of dimensions of $K$ is of particular interest. One of the most common dimensions that we are interested in is the similarity dimension which is defined as follows.

\begin{definition}
Suppose $(X,d)$ is a complete metric space and let $\mathcal{F} = \{f_{i} \colon X \to X\}_{i \in I}$ be a system of finitely many contracting similarities. The similarity dimension $d_{sim}$ is defined as the number $D$ such that 
\[
	\sum_{i \in I} c_{i}^{D} = 1.
\]
\end{definition}

In general, we know (see Falconer \cite{FA1} and McLaughlin \cite{MCLAUGH}) that the box--counting and Hausdorff dimensions of an attractor $K$ are equal and bounded above by the similarity dimension. If the system is defined on a Euclidean space and the images of the attractor under the maps $f_{i}$ do not overlap too much, then the Assouad dimension also equals the box--counting dimension. An example of such a property is the weak separation condition, which was introduced by Zerner \cite{ZE}, in the context of a Euclidean space. 

\begin{definition}\label{normal weak separation}
Suppose that $\mathcal{F} = \{f_{i} \colon \mathbb{R}^{s} \to \mathbb{R}^{s} \}$ is a system of contracting similarities, with $K$ as an attractor. We say that the system satisfies the weak separation property if there exists $\epsilon >0$ such that for any given $0 < b< 1$ and any $\aaa, \bb \in I_{b}$, we have
\[
	f_{\aaa} = f_{\bb} \qquad \text{or} \qquad \|f_{\aaa}^{-1} f_{\bb} - i_{s}\|_{L^{\infty}(K)} \geq \epsilon,
\]
where $i_{s}$ denotes the identity map $i_{s} \colon \mathbb{R}^{s} \to \mathbb{R}^{s}.$
\end{definition}

Fraser, Olson, Robinson \& Henderson \cite{FORX} used the notion of Ahlfors regularity and proved that the Assouad dimension also coincides with the Hausdorff and box--counting dimensions, under the weak separation condition. In Section \ref{zerner weak separation}, we give an independent proof of this result, without using Ahlfors regularity, solely based on the definitions and the separation condition. Moreover, the proof provides us a useful model for the analysis of sets of differences that follows.

Even in this simple case, there are examples due to Henderson \cite{Henderson} of sets with small Assouad dimension but maximal Assouad dimension for the set of differences. Motivated by the work of Olson \& Robinson \cite{OR}, \cite{Rob} we are interested in establishing non-trivial bounds for the Assouad dimension of differences of self-similar sets in terms of the Assouad dimension of the set itself. In Section \ref{wspdd}, we introduce a new separation condition which we call the weak separation condition for differences and is defined as follows.

\begin{definition}
Suppose that $\mathcal{F} = \{f_{i} \colon \mathbb{R}^{s} \to \mathbb{R}^{s}\}$ is  an system of contracting similarities. Suppose that $K$ is the attractor of the system. The system satisfies the weak separation condition for differences if there exist $M, \epsilon >0$ and a collection of points $\{x_{j}\}_{j=0}^{M} \in K$ such that for every given $0<b<1$ and every $\aaa, \bb, \cc, \dd \in I_{b},$  we have 
\[
	f_{\aaa}(K) - f_{\bb}(K) = f_{\cc}(K) - f_{\dd}(K)
\]
 or
\[
	\|f_{\aaa}(x_{i}) - f_{\bb}(x_{j}) - f_{\cc}(x_{i}) + f_{\dd}(x_{j})\| \geq \epsilon b,
\]
for some $i, j \leq M$ that depend on $\aaa, \bb, \cc, \dd \in I_{b}.$ 
\end{definition}

We show that if the system satisfies the above condition, then the Assouad dimension of the set of differences is bounded above by twice the dimension of the set.

Finally, in Section \ref{cantor sets}, we consider Cantor sets, which are the simplest example of self-similar fractals. We show that symmetric Cantor sets and a particular class of asymmetric Cantor sets satisfy the weak separation condition for differences. In particular, we obtain non-trivial bounds for the Assouad dimension of Cantor sets that fall into this class.

\section{Systems satisfying the weak separation condition}\label{zerner weak separation}

In this section, we give an alternative proof of the fact that when a system of contracting similarities satisfies the weak separation property (Definition \ref{normal weak separation}), then the Assouad dimension coincides with the box--counting dimension. We first recall the definition of an affine space and of vectors in general position.

\begin{definition}
Suppose $k,s \in \mathbb{N}$ with $k \leq s+1$ and $x_{1},\dots,x_{k} \in \mathbb{R}^{s}$. The affine space generated by $\{x_{1},\dots,x_{k}\}$, $A(x_{1},\dots,x_{k})$ is the collection of all points of the form
$$\sum_{j=1}^{k} a_{j} x_{j} \, \, \, \, \, \text{and} \, \, \, \, \, \sum_{j=1}^{k} a_{j} = 1.$$
\end{definition}

\begin{definition}
We say that $\{x_{j}\}_{j=0}^{s} \subset \mathbb{R}^{s}$ are in general position if no $x_{i}$ lies in the affine space generated by any subcollection of the $\{x_{j}\}$ consisting of less or equal than $s$ points. In other words, no $m$ of them can lie in an $(m-1)$-- dimensional hyperplane for $m \leq s$.
\end{definition}
It is easy to see that if $\{x_{j}\}_{j=0}^{s} \subset \mathbb{R}^{s}$ are in general position, then $A(x_{1},\dots,x_{k}) = \mathbb{R}^{s}$. Moreover, the vectors $\{x_{j} - x_{0}\}_{j=1}^{s} \subset \mathbb{R}^{s}$ are linearly independent and they span the whole space.

We now state and prove the following general lemma, which will give an equivalent property with the weak separation condition.

\begin{lemma}\label{BG}
Suppose $K \subset \mathbb{R}^{s}$ is compact. Then, for every $\{x_{j}\}_{j=0}^{s} \subset K$ in general position there exists an $C >0$ such that for every affine map $h \colon \mathbb{R}^{s} \to \mathbb{R}^{s}$ of the form
$$h(x) = A  x + b,$$ where $b$ is a constant and $A$ is an $s \times s $ matrix, we have
$$\|h - I\|_{L^{\infty}\left( K \right)} \leq C |h(x_{j_{*}}) - x_{j_{*}}|,$$
for some $j_{*} \in \{0,1, \dots ,s\}$
that depends on $h$.
\end{lemma}

\begin{proof}
Let $j_{*}$ be such that
\[
	|h(x_{j_{*}}) - x_{j_{*}}| = \max\{|h(x_{j}) - x_{j}| : j=0,\ldots,s\}
\]
Let also $x \in K$ be such that
\[
    \|h-I\|_{L^{\infty}\left([0,1]^{s}\right)} = |h(x) - x|.
\]
Since $A(x_{0},\dots,x_{s}) = \mathbb{R}^{s}$, we choose $\{a_{j}\}_{j=0}^{s} \in \mathbb{R}$ such that
\[
	x = \sum_{j=0}^{s}a_{j}x_{j} \, \, \, \, \text{and} \, \, \, \,  \sum_{j=0}^{s}a_{j} = 1.
\]
Consequently,

\begin{align*}
|h(x) - x| & = \left|b + \sum_{j=0}^{s} a_{j} (A x_{j}  - x_{j} )\right| \\
& = \left|b + \sum_{j=0}^{s} a_{j} (h(x_{j}) - x_{j}) - b \sum_{j=0}^{s} a_{j}\right| \\
& \leq \max \{|h(x_{j}) - x_{j}| : j=0,...,s\} \sum_{j=0}^{s}|a_{j}|.
\end{align*} 
It remains to estimate  $\sum_{j=0}^{s}|a_{j}|$.

To do this, we make the following computation:
\begin{align*}
x - x_{0} & = \sum_{j=0}^{s} a_{j}x_{j} - x_{0} =  \sum_{j=1}^{s} a_{j}x_{j} + \left(1 - \sum_{j=1}^{s} a_{j}\right) x_{0} - x_{0}\\
&  =  \sum_{j=1}^{s} a_{j}(x_{j} - x_{0}).
\end{align*}
Since $\{x_{j} - x_{0}\}_{j=1}^{s}$ forms a basis for $\mathbb{R}^{s}$, the quantity $\sum_{j=1}^{s} |a_{j}|$ is a norm of the vector $x_{j} - x_{0} \in \mathbb{R}^{s}.$ Hence, there exists a constant $C_{1} >0$, which is independent of $x$  such that
$$\sum_{j=1}^{s} |a_{j}| \leq C_{1} |x - x_{0}| \leq C_{1} \, \mathrm{diam}(K).$$

Moreover,
\begin{align*}
|a_{0}| = \left| 1 - \sum_{j=1}^{s} a_{j} \right| \leq 1 + \sum_{j=1}^{s} |a_{j}| \leq 1 + C_{1} \, \mathrm{diam}(K).
\end{align*}
All in all we deduce that
$$\sum_{j=0}^{s} |a_{j}| \leq 	C_{2},$$
where $C$ is independent of $x$. 

All in all, we obtain
\begin{align*}
	\|h-I\|_{L^{\infty}\left(K\right)} \leq 
	C|h(x_{j_*}) - x_{j_*}|,
\end{align*}
where $C$ is independent of $h$.
\end{proof}
We now have the following Corollary.
\begin{corollary}\label{BG1}
Suppose that the IFS satisfies the weak separation condition. Let also $K$ be the attractor of the system Then, for every $\{x_{j}\}_{j=0}^{s} \subset K$ in general position, there exists an $M > 0$ depending only on the $\{c_{i}\}_{i=1}^{|\mathcal{I}|}$ and $\{x_{j}\}_{j=0}^{s}$ such that
\[|(f_{\alpha} -f_{\beta})(x_{j})| \geq \epsilon M r ,\]
for some $j \in \{0,\ldots,s\},$ which depends on $\alpha, \beta \in I_{r}$.
\end{corollary}

\begin{proof}
Take
	$h = f_{\aaa}^{-1} f_{\bb}$ in Lemma \ref{BG}.
	Then, for $\aaa, \bb \in I_{r}$, let $j \leq s$ be such that
	$$|f_{\aaa}^{-1} f_{\bb}(x_{j}) - x_{j}| = |f_{\aaa}^{-1} f_{\bb}(x_{j}) - f_{\aaa}^{-1} f_{\aaa}(x_{j})| \geq M \|h - I\|_{L^{\infty}\left(K\right)}.$$
From the weak separation condition (Definition \ref{normal weak separation}) we immediately deduce that
	$$|(f_{\alpha} -f_{\beta})(x_{j})| \geq \epsilon M r,$$
	for some $M >0.$
	\end{proof}

Before we proceed to the proof of the main result of this Section, we want to introduce some terminology from graph theory, which will be useful in what follows. 

\begin{definition}
We define an undirected graph as an ordered pair $G=(V,E)$, where $V$ is a set of vertices and $E$ is a set of edges, which are unordered pairs of vertices.
\end{definition}

\begin{definition}
We say that an undirected graph $G=(V,E)$, with $n$ vertices is complete if every two vertices are connected with a unique edge. 
\end{definition}

\begin{definition}
An $r$-colouring of the edges of a graph $(V,E)$ is a function $\mathfrak{g} \colon E \to \{1,2,\cdots,r\}.$
\end{definition}

We now state a version of Ramsey's theorem. For a more detailed analysis of Ramsey theory, see Chapter 1 in the book of Katz \& Reimann \cite{katzramsey}.

\begin{theorem}[Ramsey's Theorem]
Suppose that we have $r$ colours and $(n_{1},\cdots,n_{r})$ integers. Then, there exist a number $R(r, n_{1}, n_{2}, \cdots, n_{r})$ such that if $G$ is a complete graph with at least $R(r, n_{1}, n_{2}, \cdots, n_{r})$ vertices, there exists an $i \in \{1,\dots,r\}$ and a complete subgraph $T$ of $G$ of order $n_{i}$ such that all the edges in $T$ are coloured with the colour $i$. 
\end{theorem}

An immediate corollary is the following.
\begin{corollary}\label{Ramsey}
Suppose that $G$ is a complete graph and suppose $N \in \mathbb{N}.$ Suppose also that we have an $r$-colouring of the edges of $G$. If every monochromatic complete subgraph of $G$ has order at most $N$, then
\[
	|G| < R(r,N+1,...,N+1).
\]
\end{corollary}

\begin{proof}
Suppose that $|G| \geq R(r,N+1,...,N+1)$. Then, by Ramsey's theorem, there exists a complete monochromatic subgraph of order $N+1$, which violates the hypothesis.
\end{proof}

We now show directly that when the IFS satisfies the weak separation property then the Assouad dimension of the attractor equals its box--counting dimension. In particular, since $\romd_{A}(K) \geq \romd_{B}(K)$, we only need to prove the upper bound. Note that the Hausdorff dimension is also equal in this case, since $\romd_{B}(K) = \romd_{H}(K)$ (see Falconer \cite{FA}). The proof provides a useful argument for the more involved analysis of sets of differences which follows in the next section.

\begin{theorem}
Suppose that $\mathcal{F} = \{f_{i} \colon \mathbb{R}^{s} \to \mathbb{R}^{s} \}$ is an iterated function system that satisfies the weak separation property. Let $K$ be the attractor of the system and suppose that $K$ is not contained in a hyperplane. Then,
$$\romd_{A}(K) = \romd_{B}(K).$$
\end{theorem}

\begin{proof}
Let $d > \romd_{B}(K).$ 
Suppose, wlog that
$$K \subset B_{1}(0).$$
Then, for any $y \in K$, we have
$$K \subset B_{2}(y).$$
Suppose $x \in K$ and $r >0$. 
Let 
\[
	G_{r}(x) = \{f_{\aaa} : \aaa \in I_{r}, B_{r}(x) \cap f_{\aaa}(K) \neq \emptyset\}.
\]
Then, for any $y \in K$, we have 
\begin{align*}
	B_{r}(x) \cap K & = \cup_{\aaa \in I_{r}} B_{r}(x) \cap f_{\aaa}(K)\\
	& = \cup_{f_{\aaa} \in G_{r}(x)} B_{r}(x) \cap f_{\aaa}(K)\\
	& \subset \cup_{f_{\aaa} \in G_{r}(x)} B_{r}(x) \cap f_{\aaa}(B_{2}(y)),
\end{align*}
which implies that
\begin{equation}\label{cover of K}
B_{r}(x) \cap K \subset \cup_{f_{\aaa} \in G_{r}(x)} B_{r}(x) \cap B_{2r}(f_{\aaa}(y)),
\end{equation}
for all $y \in K$.
We claim that we can bound the cardinality of $G_{r}(x)$ independently of $r,x$. 
Since $K$ is not contained in a hyperplane, there exist $\{x_{j}\}_{j=0}^{s} \subset K$ in general position. By Lemma \ref{BG1}, there exists an $\epsilon >0$ such that for every choice of $f_{\aaa} ,f_{\bb} \in G_{r}(x)$, there exists a $j \leq s$ such that
\begin{equation}\label{WSPR}
|f_{\aaa}(x_{j}) - f_{\bb}(x_{j})| \geq \epsilon r.
\end{equation}
Let 
\[
	T_{r}(x) = \{0 \leq j \leq s : |f_{\aaa}(x_{j}) - f_{\bb}(x_{j})| \geq \epsilon r, \qquad \text{for some} \qquad f_{\aaa}, f_{\bb} \in G_{r}(x).\}
\]
Obviously, $|T_{r}(x)| \leq s+1$, for all $r,x$. We now consider $G_{r}(x)$ as an unordered graph with vertices $f_{\aaa}$ and edges $E	= \{\{f_{\aaa},f_{\bb}\} : f_{\aaa}, f_{\bb} \in G_{r}(x)\}$. For each edge $\{f_{\aaa},f_{\bb}\}$, we assign a colour $j \in T_{r}(x)$ such that
$$|f_{\aaa}(x_{j}) - f_{\bb}(x_{j})| \geq \epsilon r.$$ 

Suppose that $P^{j}_{r}(x)$ is a complete monochromatic subgraph of $G_{r}(x)$, of color $j \leq s+1$.
Then, for every $f_{\aaa} \in P^{j}_{r}(x)$, we have
\begin{align*}
	B(x,r) \cap f_{\aaa}(F) \neq \emptyset & \Rightarrow B_{r}(x) \cap f_{\aaa}(B_{2}(x_{j})) \neq \emptyset \\
	& \Rightarrow B_{r}(x) \cap B_{2r}(f_{\aaa}(x_{j})) \neq \emptyset,
\end{align*}
which implies that
\begin{equation}\label{MS1}
|f_{\aaa}(x_{j}) - x| \leq 3r.
\end{equation}
Moreover, for any $f_{\aaa},f_{\bb} \in P^{j}_{r}(x),$ we have by definition 
\begin{equation}\label{MS2}
	|f_{\aaa}(x_{j}) - f_{\bb}(x_{j})| \geq \epsilon r.
\end{equation}
In particular $f_{\aaa}(x_{j}) \neq f_{\bb}(x_{j}),$ for all $f_{\aaa}, f_{\bb} \in P^{j}_{r}(x).$ Consequently, in order to count the number of vertices in $P^{j}_{r}(x)$, it suffices to count the points $f_{\aaa}(x_{j})$, for $f_{\aaa} \in P^{j}_{r}(x)$. By \eqref{MS2}, the balls of radius $\epsilon r /2$, with centres $f_{\aaa}(x_{j}),$ for $f_{\aaa} \in P^{j}_{r}(x)$ are disjoint and by \eqref{MS1}, all the centres lie in a ball of radius $3r$, centred at $x$.
Thus,
\[
\bigcup_{f_{\aaa} \in P^{j}_{r}(x)} B_{\frac{\epsilon r}{2}} (f_{\aaa}(x_{j})) \subseteq B_{3r +\epsilon r}(x).
\]
Therefore, if $\mu$ is the $s$-dimensional Lebesque measure, we have
\[
|P^{r}_{j}(x)| \leq \frac{\mu\left( B_{3r +\epsilon r}(x) \right)}{\mu \left( B_{\frac{\epsilon r}{2}} \right) } = M',
\]
which is independent of $r,x.$
Since $G_{r}(x)$ is a complete graph and we bounded the order of any complete monochromatic subgraph independently of $r,x$, we have by Corollary \ref{Ramsey} that
$$|G_{r}(x)| \leq  M,$$
independent of $r,x$.

We now enumerate $G_{r}(x)$ using the following parametrisation.
$$G_{r}(x) = \{f_{\aaa_{k}}\}_{k=1}^{M}.$$
Now, let $N = N(K, \rho/r)$ denote the number of balls of radius $\rho/r$ required to cover $K$.
Let the centres of those balls be $y_{j}$, for $j \leq N.$ Then, by \eqref{cover of K}, we have
\begin{align*}
B_{r}(x) \cap K & \subseteq \cup_{k=1}^{M} f_{\aaa_{k}} (K)\\
& \subset \cup_{k=1}^{M} f_{\aaa_{k}} \left(\cup_{j=1}^{N} B_{\rho/ r}(y_{j})\right)\\
& = \cup_{k=1}^{M} \cup_{j=1}^{N} B_{\rho}(f_{\aaa_{k}} (y_{j}))
\end{align*} We know by definition of the box--counting dimension that there exists some constant $C >0$ such that
$$N \leq C \left(\frac{r}{\rho}\right)^{d}.$$ 
Thus,
$$N_{K}(r,\rho) \leq M N \leq M C \left(\frac{r}{\rho}\right)^{d}.$$
Therefore, $d \geq \romd_{A}(K)$ and since $d > \romd_{B}(K)$ was arbitrary we have
$\romd_{A}(K) \leq \romd_{B}(K).$
\end{proof}

\section{The weak separation condition for differences}\label{wspdd}

In this section, we study differences of attractors of Iterated Function Systems in Euclidean spaces. We want to establish non-trivial bounds for the Assouad dimension of the set of differences in terms of the Assouad dimension of the attractor. In particular, we show that under a suitable separation condition, the Assouad dimension of $K-K$ is bounded above by twice the Assouad dimension of $K$. Note that such non-trivial bounds do not hold in general as there are examples on the real line due to Henderson \cite{Henderson} where $\romd_{A}(K) <\epsilon$, for any $\epsilon >0$ and $\romd_{A}(K-K) =1$.

\begin{definition}
Suppose that $\mathcal{F} = \{f_{i} \colon \mathbb{R}^{s} \to \mathbb{R}^{s}\}$ is a system of contracting similarities. Suppose that $K$ is the attractor of the system. The IFS satisfies the weak separation condition for differences if there exist $M, \epsilon >0$ and a collection of points $\{x_{j}\}_{j=0}^{M} \in K$ such that for every given $0<b<1$ and every $\aaa, \bb, \cc, \dd \in I_{b},$  we have 
\[
	f_{\aaa}(K) - f_{\bb}(K) = f_{\cc}(K) - f_{\dd}(K)
\]
 or
\[
	\left|f_{\aaa}(x_{i}) - f_{\bb}(x_{j}) - f_{\cc}(x_{i}) + f_{\dd}(x_{j})\right| \geq \epsilon b,
\]
for some $i, j \leq M$ that depend on $\aaa, \bb, \cc, \dd \in I_{b}.$ 
\end{definition}

We also formulate a stronger separation condition, which involves the definition of Hausdorff distance, whose definition we now recall.
\begin{definition}
Suppose $(X,d)$ is a metric space and let $A,B$ be compact subsets of $X$. Then the Hausdorff distance is defined as
\[
	d_{H}(A,B) = \max\{\dist(A,B), \dist(B,A)\}.
\]
\end{definition}

We prove that the weak separation for differences is satisfied if for any scale $b <1$, the sets $f_{\aaa}(K) - f_{\bb}(K)$, $f_{\cc}(K) - f_{\dd}(K)$ for $\aaa, \bb, \cc, \dd \in I_{b}$ are either equal or their Hausdorff distance is bounded away from zero. 
\begin{lemma}\label{sufficient condition for weak separation}
Suppose that $\mathcal{F} = \{f_{i} \colon \mathbb{R}^{s} \to \mathbb{R}^{s}\}$ is a system of contracting similarities. Suppose that there exists a $\zeta >0$ such that for any given $0<b<1$ we have that either
\[
	f_{\aaa}(K) - f_{\bb}(K) = f_{\cc}(K) - f_{\dd}(K)
\]
 or
 \[
 	d_{H}(f_{\aaa}(K) - f_{\bb}(K), f_{\cc}(K) - f_{\dd}(K)) \geq \zeta b,
 \]
 for all $\aaa, \bb, \cc, \dd \in I_{b}.$
Then, the weak separation condition for differences is satisfied.
\end{lemma}

\begin{proof}
Let $\aaa, \bb, \cc, \dd \in I_{b}$. Suppose that
\[
	f_{\aaa}(K) - f_{\bb}(K) \neq f_{\cc}(K) - f_{\dd}(K).
\]
Let $\{x_{j}\}_{j=0}^{M}$ be an $\zeta/4$ net in $K$, i.e. 
\[
	K \subset \bigcup_{j=0}^{M} B_{\zeta/4}(x_{j}) \qquad \text{and} \qquad |x_{i} - x_{j}| \geq \frac{\zeta}{4}.
\]
Assume without loss of generality that
$$d_{H}(f_{\aaa}(K) - f_{\bb}(K), f_{\cc}(K) - f_{\dd}(K)) = \dist(f_{\aaa}(K) - f_{\bb}(K), f_{\cc}(K) - f_{\dd}(K)) \geq \zeta b.$$
Using the compactness of $K$, let $x,y \in K$ be such that
$$\dist(f_{\aaa}(K) - f_{\bb}(K), f_{\cc}(K) - f_{\dd}(K)) = \dist(f_{\aaa}(x) - f_{\bb}(y), f_{\cc}(K) - f_{\dd}(K)).$$
Let $i,j \leq M$ be such that
$$|x - x_{i}| \leq \frac{\zeta}{4} \qquad \text{and} \qquad |y-x_{j}| \leq \frac{\zeta}{4}.$$
Again by the compactness of $K$ suppose that $s,t \in K$ are such that
$$\dist(f_{\aaa}(x) - f_{\bb}(y), f_{\cc}(K) - f_{\dd}(K)) = |f_{\aaa}(x) - f_{\bb}(y) - f_{\cc}(s) + f_{\dd}(t)|.$$
Then, we deduce that
\begin{align*}
|f_{\aaa}(x_{i}) - f_{\bb}(x_{j}) - f_{\cc}(x_{i}) + f_{\dd}(x_{j})| & \geq \dist(f_{\aaa}(x_{i}) - f_{\bb}(x_{j}), f_{\cc}(K) - f_{\dd}(K))\\
& = |f_{\aaa}(x_{i}) - f_{\bb}(x_{j}) - f_{\cc}(s) + f_{\dd}(t)|,
\end{align*}
which implies that
\begin{align*} 
|f_{\aaa}(x_{i}) - f_{\bb}(x_{j}) - f_{\cc}(x_{i}) + f_{\dd}(x_{j})| & \geq |f_{\aaa}(x) - f_{\bb}(y) - f_{\cc}(s) + f_{\dd}(t)|\\
& - |f_{\aaa}(x_{i}) - f_{\bb}(x_{j}) - f_{\aaa}(x) + f_{\bb}(y)|\\
& \geq \zeta b - 2 \frac{\zeta b}{4} = \frac{\zeta b}{2}. 
\end{align*}

By taking $\epsilon = \zeta/2$, the proof is complete.
\end{proof}
It is an open question whether the above condition is actually equivalent with the weak separation condition for differences.
We now state and prove the main result of this section. 

\begin{theorem}\label{Differences}
Suppose that $\mathcal{F}= \{f_{i} \colon \mathbb{R}^{s} \to \mathbb{R}^{s}\}$ is a system of contracting similarities and let $K$ be the attractor of the system. If the IFS satisfies the weak separation for differences then 
\begin{equation}\label{AS-2}
\romd_{A}(K-K) \leq 2 \romd_{A}(K).
\end{equation}
\end{theorem}

\begin{proof}
The argument is similar to the argument of the previous section. 
We use a Ramsey theory argument to prove that given any $0 < r<1$ and $z \in K-K$, the cardinality of set of maps $(f_{\aaa}, f_{\bb})$ such that 
\[
	B_{r}(z) \cap \left(f_{\aaa}(K) - f_{\bb}(K)\right) \neq \emptyset
\]
is independent of $r,z$. 

Assume without loss of generality that
$$K - K \subseteq B_{1}(0).$$
Let $d =\romd_{A}(K)$ and let also $r, \rho$ be such that $0<\rho <r<1$.
Now, we fix $z \in K-K$.

Note that for any $x \in K$, we have 
$$ K \subseteq B_{1}(x).$$

We now define the following equivalence relation. For all $\aaa,\bb,\cc,\dd \in I_{r}$ we have
$$(f_{\aaa},f_{\bb}) \sim (f_{\cc}, f_{\dd}) \Leftrightarrow f_{\aaa}(K) - f_{\bb}(K) = f_{\cc}(K) - f_{\dd}(K).$$
For $\aaa, \bb \in I_{r}$ we also define 
$$[(f_{\aaa}, f_{\bb})] = \{(f_{\cc}, f_{\dd}) : \cc ,\dd \in I_{r}, (f_{\cc}, f_{\dd}) \sim (f_{\aaa}, f_{\bb})\},$$
the equivalence class of $(f_{\aaa},f_{\bb})$.
Let
$$H_{r}(z) = \{[(f_{\aaa}, f_{\bb})] : \aaa,\bb \in I_{r}, \, B_{r}(z) \cap \left(f_{\aaa}(K) - f_{\bb}(K)\right) \neq \emptyset\}.$$
We also let $G_{r}(z)$ to be a complete set of representatives from the equivalence classes in $H_{r}(z)$. In particular, we assign a unique element $(f_{\aaa},f_{\bb}) \in G_{r}(z)$ to each equivalence class $[(f_{\aaa},f_{\bb})]$ in $H_{r}(z)$.
Thus, for  $(f_{\aaa},f_{\bb}), (f_{\cc},f_{\dd}) \in G_{r}(z)$, we have
\begin{equation}\label{property of grz}
f_{\aaa}(K) - f_{\bb}(K) \neq f_{\cc}(K) - f_{\dd}(K).
\end{equation}

We now observe by properties of $K$ (see Proposition \ref{properties of K}) that
\begin{align*}
B_{r}(z) \cap (K-K) & \subseteq B_{r}(z) \cap \left(\bigcup_{(\aaa, \bb) \in I_{r} \times I_{r}} f_{\aaa}(K) - f_{\bb}(K)\right) \\
& = \bigcup_{(f_{\aaa}, f_{\bb}) \in G_{r}(z)} B_{r}(z) \cap (f_{\aaa}(K) - f_{\bb}(K))\\
& \subseteq \bigcup_{ (f_{\aaa}, f_{\bb}) \in G_{r}(z)} B_{r}(z) \cap (B_{r} (f_{\aaa}(x))\cap K - B_{r}(f_{\bb}(y)) \cap K)\\
& = \bigcup_{(f_{\aaa}, f_{\bb}) \in G_{r}(z)} B_{r}(z) \cap B_{2r}(f_{\aaa}(x) - f_{\bb}(y)),
\end{align*}
for all $x, y \in K$.

We now claim that we can bound the cardinality of $G_{r}(z)$ independently of $r,z$. Indeed, by \eqref{property of grz}, using the weak separation property, we can find $\{x_{j}\}_{j=0}^{M} \subset K$ such that for each choice of $(f_{\aaa}, f_{\bb}), (f_{\cc}, f_{\dd}) \in G_{r}(z)$, we can find $i, j \leq M$ such that 
\begin{equation}\label{CFG1}
|f_{\alpha}(x_{i}) -f_{\beta}(x_{j}) - f_{\gamma}(x_{i}) + f_{\delta}(x_{j})| \geq \epsilon r.
\end{equation}
Based on the above, we interpret $G_{r}(z)$ as a graph and we say that an edge $\{(f_{\aaa}, f_{\bb}), (f_{\cc}, f_{\dd})\}$ is assigned a colour $(i,j)$, if
\begin{equation}\label{CFG}
|f_{\alpha}(x_{i}) -f_{\beta}(x_{j}) - f_{\gamma}(x_{i}) + f_{\delta}(x_{j})| \geq \epsilon r >0.
\end{equation}

We claim that there exists $N$ independent of $r,z$ such that 
\[|G_{r}(z)| \leq N.\]
Let $T_{ij}$ be any complete monochromatic subgraph of $G_{r}(z)$ of color $(i,j)$. Therefore, for all $(f_{\aaa}, f_{\bb}), (f_{\cc}, f_{\dd})\in T_{ij}$,  \eqref{CFG} is satisfied for the same $x_{i}, x_{j}.$
In particular for each $(f_{\aaa}, f_{\bb}), (f_{\cc}, f_{\dd}) \in T_{ij}$ we have
\begin{equation}\label{CFG2}
f_{\aaa}(x_{i}) - f_{\bb}(x_{j}) \neq f_{\cc}(x_{i}) - f_{\dd}(x_{j}).
\end{equation}
Hence, the number of vertices in $T_{ij}$ equals the number of points $\{f_{\aaa}(x_{i}) - f_{\bb}(x_{j}) : (f_{\aaa}, f_{\bb}) \in T_{ij}\}$.
For $(f_{\aaa} , f_{\bb}) \in T_{ij} \subset G_{r}(z)$, we also have
\begin{align*}
B_{r}(z) \cap (f_{\aaa}(F) - f_{\bb}(F)) \neq \emptyset & \Rightarrow B_{r}(z) \cap (B_{r}(f_{\aaa}(x_{i})) - B_{r}(f_{\bb}(x_{j}))) \neq \emptyset\\
& \Leftrightarrow B_{r}(z) \cap (B_{2r}(f_{\aaa}(x_{i}) - f_{\bb}(x_{j}))) \neq \emptyset.
\end{align*}
Therefore, we deduce that

\[|f_{\aaa}(x_{i}) - f_{\bb}(x_{j}) - z| \leq 3r,\]
and we also know that 
\[|f_{\alpha}(x_{i}) -f_{\beta}(x_{j}) - f_{\gamma}(x_{i}) + f_{\delta}(x_{j})| \geq \epsilon r.\]

Therefore, all the balls of radius $ \epsilon r /2$ and centres $f_{\aaa}(x_{i}) - f_{\bb}(x_{j})$, for $(f_{\aaa}, f_{\bb}) \in T_{ij}$ are disjoint and all the centres lie in a ball of radius $3r$ around $z$. 
It is immediate from \eqref{CFG2} that
\[|T_{j}| \leq \frac{\mu(B_{3r + \epsilon r}(z))}{\mu(B_{\frac{\epsilon r}{2}}(f_{\alpha}(x_{i}) -f_{\beta}(x_{j})))} \leq N_{1}, \]
independent of $r,z,(i,j)$.
Hence, by Ramsey's Theorem, we have that
$$|G_{r}(z)| \leq N,$$
independent of $r,z$.

Now, we enumerate $G_{r}(z)$ using the following parametrisation
$$G_{r}(z) = \{(f_{\aaa_{k}}, f_{\bb_{k}})\}_{k=1}^{N}.$$
Take any $x \in K$. Then, we have
\begin{align*}
B_{r}(z)\cap(K-K) & \subseteq \bigcup_{k=1}^{N} B_{r}(z) \cap (f_{\aaa_{k}}(K) - f_{\bb_{k}}(K)) \\
& \subseteq \bigcup_{k=1}^{N} B_{r}(z) \cap (B_{r}(f_{\aaa_{k}}(x)) \cap K - B_{r}(f_{\bb_{k}}(x)) \cap K).
\end{align*}
Since $f_{\aaa_{k}}(x), f_{\bb_{k}}(x) \in K,$ we can cover each of these balls centred at those points by $N'= N_{K}(r, \rho/2)$ balls of radius $\rho/2$ centred at $K$. Let the centres of those balls be $z^{k}_{i}$.
Then,
\begin{align*}
B_{r}(z)\cap(K-K) & \subseteq \bigcup_{k=1}^{N} \left( \bigcup_{i=1}^{N'} B_{\frac{\rho}{2}} (z^{k}_{i}) - \bigcup_{j=1}^{N'} B_{\frac{\rho}{2}} (z^{k}_{j})\right)\\
& \subseteq  \bigcup_{k=1}^{N} \bigcup_{i, j=1}^{N'} B_{\rho}(z^{k}_{i} - z^{k}_{j})
\end{align*}
Thus,
\[
	N_{K-K}(r,\rho) \leq N (N')^{2} \leq N C \left(\frac{r}{\rho}\right)^{2d}.\qedhere
\]
\end{proof}

\section{Differences of Cantor sets}\label{cantor sets}

Cantor sets are one of the most common examples of self-similar fractals. They are constructed by an iterated process of removing intervals from the unit interval $[0,1]$. We first focus on symmetric Cantor sets, where at each stage of the iteration the intervals that remain are of the same length. We will show that symmetric Cantor sets and a particular class of asymmetric Cantors sets satisfy the weak separation condition for differences. In particular, the Assouad dimension of differences of these Cantor sets obeys bounds in terms of the Assouad dimension of the Cantor set itself.

\subsection{Symmetric Cantor sets}

Symmetric Cantor sets are constructed by removing intervals of proportionate length from $[0,1]$ repeatedly.
In particular, let $\lambda <1/2$ and suppose that $C_{0}$ is the interval $[0,1]$. We define $C_{k+1}$ by removing intervals of length $c_{k} \lambda$, from $C_{k}$, where $c_{k}$ is the length of the intervals in $C_{k}$ (see also Figure \ref{Cantor}). Then, the middle-$\lambda$ Cantor set is defined as 

$$C = \bigcap_{k=0}^{\infty} C_{k}.$$ 

\begin{figure}[!ht]
   \includegraphics[scale=0.7]{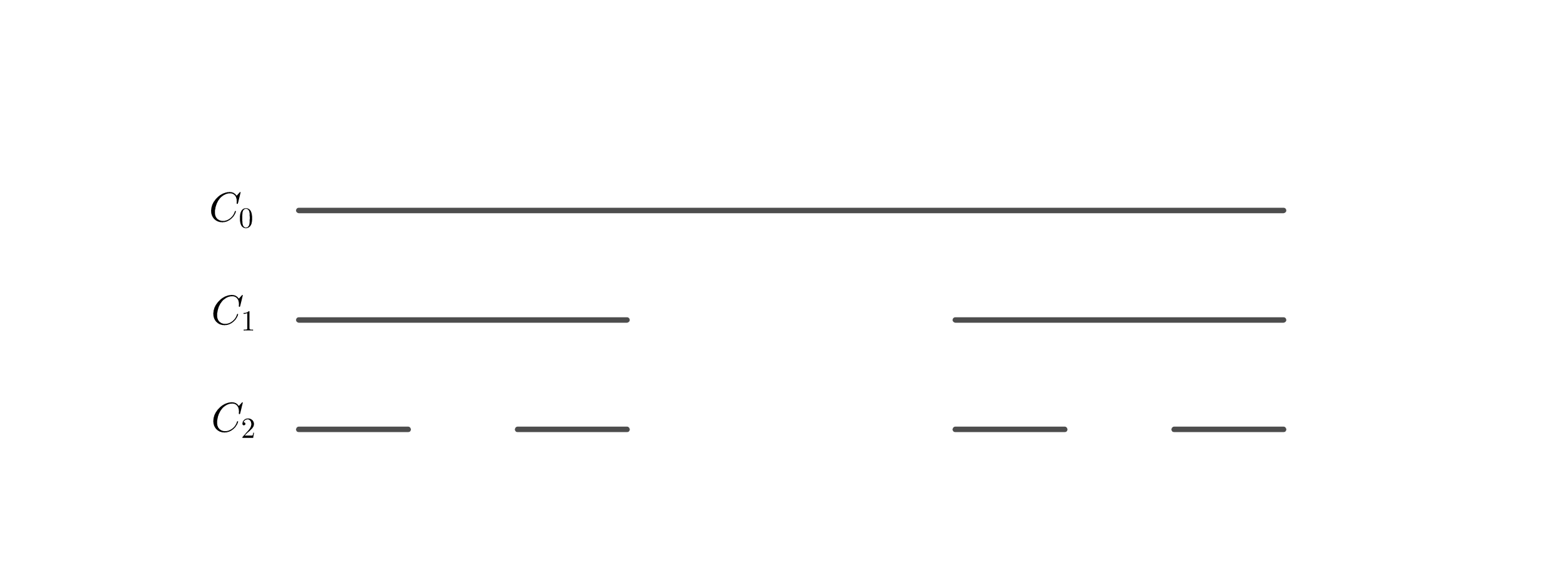}
   \caption{The first stages of the iteration for the middle--$1/3$ Cantor set.}
   \label{Cantor}
\end{figure}

A symmetric Cantor set can also be defined as the attractor of an Iterated Function system. For any $\lambda <1/2$, the middle-$\lambda$ Cantor set $\mathcal{C}_{\lambda}$, is the attractor of the iterated function system that consists of
\[
	f_{1}(x) = \lambda x \qquad \text{and} \qquad f_{2}(x) = \frac{1}{c} x + (1 - \lambda).
\]
We recall the open set condition, which holds if there exists an open set $U$ such that
\[
	U \supseteq \bigcup_{i=1}^{|I|} f_{i}(U) \qquad \text{and} \qquad f_{i}(U) \cap f_{j}(U) = \emptyset.
\]

It is easy to see (see Chapter 13 in the book of  Falconer \cite{FA}) that the Cantor set satisfies the open set condition for $U = (0,1)$ which in particular implies (see again the book of Falconer \cite{FA}) that 
\[
	\romd_{A}(C_{\lambda}) = \romd_{B}(C_{\lambda}) = \romd_{sim}(C_{\lambda}) = \frac{\log 2}{\log \frac{1}{\lambda}}.
\]
Henderson \cite{Henderson} studied the Assouad dimension of the set of differences $C_{\lambda} - C_{\lambda}$ and showed that it is strictly bounded above by twice the Assouad dimension of $C_{\lambda}.$ In particular, $C_{\lambda} - C_{\lambda}$ is an attractor of another system of similarities, which satisfies the weak separation property. We directly show that the Cantor set $C_{\lambda}$, for $\lambda < 1/3$  satisfies the weak separation property for differences, which immediately gives an example of a set that satisfies that property and 
\[
	\romd_{A}(C_{\lambda} - C_{\lambda}) < 2 \romd_{A}(C_{\lambda}).
\]
Note that the restriction on $\lambda$ is plausible since otherwise
\[
	2 \romd_{A}(C_{\lambda}) = 2 \frac{\log 2}{\log \frac{1}{\lambda}} > 1 = \romd_{A}((-1,1)) \geq \romd_{A}(C_{\lambda} - C_{\lambda}).
\]

\begin{proposition}\label{symmetric cantor sets proof}
The Cantor set $C_{\lambda}$, for $\lambda < \frac{1}{3}$ satisfies the weak separation property for differences.
\end{proposition}

\begin{proof}
Fix $\lambda < 1/3$. Then, $C_{\lambda}$ is the attractor of the Iterated Functions system
\[
	f_{1}(x) = \lambda x \qquad \text{and} \qquad f_{2}(x) = \lambda x + (1 - \lambda).
\]
Take any $0<b<1$. We claim that there exists a $\delta >0$, such that for any $\aaa, \bb, \cc, \dd \in I_{b}$ we have
\[
	f_{\aaa}(C_{\lambda}) - f_{\bb}(C_{\lambda}) = f_{\cc}(C_{\lambda}) - f_{\dd}(C_{\lambda}),
\]
or 
\[
|f_{\aaa}(x) - f_{\bb}(y) - f_{\cc}(x) + f_{\dd}(y)| \geq \delta b,
\]
for every $x,y \in C_{\lambda}$.
In particular, this obviously implies the weak separation for differences by choosing any single point in the Cantor set.

Now, we fix $0<b<1$. 
Then, it is easy to see for any $\aaa, \bb \in I_{b}$ with $\aaa=(i_{1},\cdots,i_{k}),  \, \bb = (i_{1},\cdots,i_{m})$, we have that $k=m$ and
\begin{equation}\label{coefficients}
	c_{\aaa} = c_{\bb} = \lambda^{k} \leq  b \leq \lambda^{k-1} = c_{\bar{\aaa}} = c_{\bar{\bb}}.
\end{equation}
We also have that for any $\aaa=(i_{1},\cdots,i_{k}) \in I_{b}$, there exists some translation $q_{\aaa}$ such that for any $x \in C_{\lambda}$
\[
	f_{\aaa}(x) = \lambda^{k} x + q_{\aaa}
\]
and 
$$q_{\aaa} = \sum_{i=0}^{k-1} t_{i} \, \lambda^{i} (1-\lambda)= (1-\lambda) \sum_{i=0}^{k-1} t_{i} \lambda^{i},$$
for $t_{i} \in \{0,1\}.$ 

Therefore, for any $\aaa, \bb, \cc, \dd \in I_{b}$, there exists some $k \in \mathbb{N}$ such that $|\aaa| = |\bb| = |\cc| = |\dd| = k$ and for any $x,y \in C_{\lambda}$
\[
	|f_{\aaa}(x) - f_{\bb}(y) - f_{\cc}(x) + f_{\dd}(y)| = (1-\lambda) \left|\sum_{i=0}^{k-1} a_{i} \lambda^{i} \right|,
\]
where $a_{i} \in \{-1,-2,0,1,2\}.$

Suppose now that 
\[
	f_{\aaa}(C_{\lambda}) - f_{\bb}(C_{\lambda}) \neq f_{\cc}(C_{\lambda}) - f_{\dd}(C_{\lambda}).
\]
We claim that $\left|\sum_{i=0}^{k-1} a_{i} \lambda^{i}\right| \neq 0$. 

Indeed, suppose without loss of generality that 
\[
	d_{H}(f_{\aaa}(C_{\lambda}) - f_{\bb}(C_{\lambda}), f_{\cc}(C_{\lambda}) - f_{\dd}(C_{\lambda})) = \dist(f_{\aaa}(C_{\lambda}) - f_{\bb}(C_{\lambda}), f_{\cc}(C_{\lambda}) - f_{\dd}(C_{\lambda}))
\]
and let $x_{0}, y_{0} \in C_{\lambda}$ be such that
\[
	d_{H}(f_{\aaa}(C_{\lambda}) - f_{\bb}(C_{\lambda}), f_{\cc}(C_{\lambda}) - f_{\dd}(C_{\lambda})) = \dist(f_{\aaa}(x_{0}) - f_{\bb}(y_{0}), f_{\cc}(K) - f_{\dd}(K)) > 0.
\]
Then,
\begin{align*}
	\left|\sum_{i=0}^{k-1} a_{i} \lambda^{i}\right| & = \frac{1}{(1-\lambda)}|f_{\aaa}(x_{0}) - f_{\bb}(y_{0}) -f_{\cc}(x_{0}) + f_{\dd}(y_{0})| \\
	& \geq  \frac{1}{(1-\lambda)} \dist(f_{\aaa}(x_{0}) - f_{\bb}(y_{0}), f_{\cc}(K) - f_{\dd}(K))\\
    & = \frac{1}{(1-\lambda)}  d_{H}(f_{\aaa}(C_{\lambda}) - f_{\bb}(C_{\lambda}), f_{\cc}(C_{\lambda}) - f_{\dd}(C_{\lambda})) > 0.
\end{align*}

Suppose that $\sum_{i=0}^{k-1} a_{i} \lambda^{i} > 0$. We claim that there exist $\widehat{a_{i}} > 0$ such that
\[
	\sum_{i=0}^{k-1} a_{i} \lambda^{i} = \sum_{i=0}^{k-1} \widehat{a_{i}} \lambda^{i}.
\]
We construct $\widehat{a_{i}}$ by the following process.
If $a_{k-1} \geq 0$, we set $\widehat{a_{k-1}} = a_{k-1}.$
If $a_{k-1} < 0$, we write 
$$a_{k-1} =\left(\frac{1}{\lambda} + a_{k-1}\right)\lambda^{(k-1)} - \lambda^{(k-2)}$$ and we set 
$$\widehat{a_{k-1}} = \left(\frac{1}{\lambda}+ a_{k-1}\right).$$
Then, $\widehat{a_{k-1}} > 0$, since $\lambda < \frac{1}{3}$.

Now, if $a_{k-2}\lambda^{(k-2)} - \lambda^{(k-2)} = (a_{k-2} - 1) \lambda^{k-2} < 0$, then we again write   
$$(a_{k-2} -1) \lambda^{(k-2)}  = \left(\frac{1}{\lambda} + a_{k-2} -1\right) \lambda^{(k-2)} - \lambda^{(k-3)}$$
and we set 
$$\widehat{a_{(k-2)}} = \frac{1}{\lambda} + a_{k-2} -1.$$
Then, $\frac{1}{\lambda} + a_{k-2} -1 > 0$, since $\lambda <1/3$ and we carry on this procedure until we construct $\widehat{a_{0}}$. Now, for all $1 \leq i \leq k-1$, we have that
\[
	\widehat{a_{i}} = a_{i} \, \, \, \text{or} \, \, \, (\frac{1}{\lambda} + a_{i}) \, \, \, \text{or} \, \, \, (\frac{1}{\lambda} + a_{i} -1),
\]
which are all non negative. We claim that 
\[
	A= \sum_{i=1}^{k-1} \widehat{a_{i}} \lambda^{i} < 1.
\]
We observe that for all $1 \leq i \leq k-1$, $\widehat{a_{i}} \leq 2$. 
Therefore
\[
	A \leq 2 \sum_{i=1}^{k-1} \lambda^{i} < 2 \sum_{i=1}^{k-1} \left(\frac{1}{3}\right)^{i}  \leq 2 (\frac{3}{2} - 1)  = 1.
\]
Hence, $\widehat{a_{0}} > -1$. If $\widehat{a_{0}} < 0$, then 
\[
	\widehat{a_{0}} \leq -1,
\] 
a contradiction.

By a symmetric argument, i.e. by subtracting $c$ where necessary we have that if the sum is negative then it can be written such that all the coefficients are non--positive. Assume that $\widehat{a_{i}}  \geq 0$, for all $i$. In particular, by the construction above, we observe that if $\widehat{a_{i}}>0$, then
\begin{equation}\label{positive coefficients}
	\widehat{a_{i}} \geq \left(\frac{1}{\lambda} - 3\right) > 0.
\end{equation}
Let $0 \leq m \leq k-1$ such that $\widehat{a_{m}} >0$. Then, by \eqref{coefficients}, \eqref{positive coefficients}, we have
\begin{align*}
	|f_{\aaa}(x) - f_{\bb}(y) - f_{\cc}(x) + f_{\dd}(y)| & = (1-\lambda) \left|\sum_{i=0}^{k-1} a_{i} \lambda^{i} \right| = (1-\lambda) \sum_{i=0}^{k-1} \widehat{a_{i}} \lambda^{i}\\
	& \geq (1-\lambda) a_{m} \lambda^{m}  \geq (1-\lambda)\left(\frac{1}{\lambda}-3\right) \lambda^{k-1} \\
	& \geq (1-\lambda)\left(\frac{1}{\lambda}-3\right) b,
\end{align*}
which concludes the proof.
\end{proof}

We note that we can actually make all the coefficients in the above construction either negative or positive, depending on whether the sum is negative or positive. We will need this in the following section, where we prove that a particular class of Asymmetric Cantor sets satisfies the weak separation condition for differences. 

\subsection{Asymmetric Cantor sets}\label{non symmetric cantor sets}
Asymmetric Cantor sets are constructed by an iterative process of removing intervals of different lengths from the unit interval. In particular let $c_{1}, c_{2} \in (0,1)$ such that $c_{1} + c_{2} <1$. Suppose that $C_{0} = [0,1]$. We construct $C_{1}$ by removing an interval of length $1 - c_{1} - c_{2}$ from $C_{0}$ and we set $C_{1}$ to be the remaining two intervals. We carry on by removing intervals of proportionate length from each of the intervals in $C_{k}$ (see also Figure \ref{Cantorset}). 

\begin{figure}[!ht]
   \includegraphics[scale=0.7]{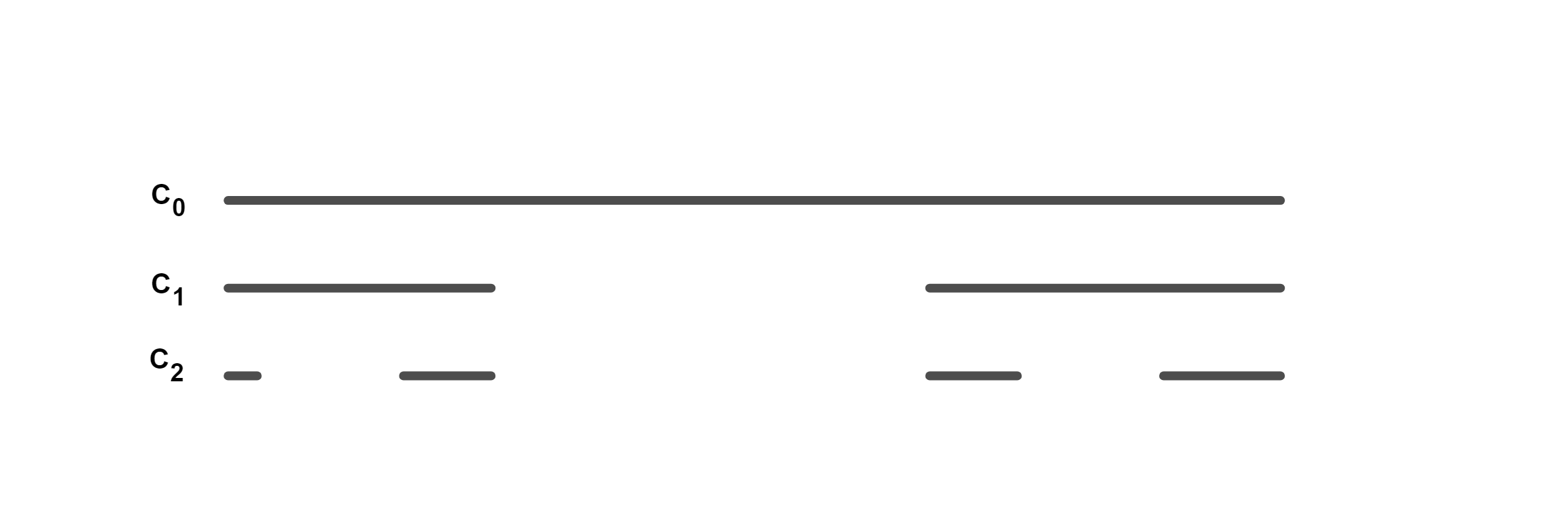}
   \caption{The first stages of the iteration for the asymmetric Cantor set.}
   \label{Cantorset}
\end{figure}

Asymmetric Cantor sets are also the attractors of the following Iterated Functions system
\[
	f_{1}(x) = c_{1} x \qquad \text{and} \qquad f_{2}(x) = c_{2} x +(1-c_{2}),
\] 
for $c_{1}, c_{2} \in (0,1)$, such that $c_{1} + c_{2} <1$. We denote the Cantor set by $K_{c_{1}c_{2}}$. It has been proven by Henderson that if $\frac{\log c_{1}}{\log c_{2}}$ is an irrational number, then $\romd_{A}(K_{c_{1} c_{2}} - K_{c_{1}c_{2}})=1$, which is maximal for this set.

It is an open question whether we can show that the Assouad dimension of $K-K$ is bounded by twice the Assouad dimension of $K$ when $\frac{\log c_{1}}{\log c_{2}}$ is any rational number. 

In this section, we show that if $\frac{\log c_{1}}{\log c_{2}}$ is a rational number and $c_{1} < c_{2} < \frac{1}{4}$, then the weak separation for differences is satisfied. In particular, we prove the following theorem, which is the main result of this section. 

\begin{theorem}
Suppose $c_{1} < c_{2} < 1/4$ such that $\frac{\log c_{1}}{\log c_{2}}$ is rational. Suppose that $K_{c_{1} c_{2}}$ is the attractor of the system $\mathcal{F} = \{f_{1}, f_{2}\}$ such that 
\[
	f_{1}(x) = c_{1} x \qquad \text{and} \qquad f_{2}(x) = c_{2} x + (1 - c_{2}).
\]
Then,
\[
	\romd_{A}(K_{c_{1}c_{2}} - K_{c_{1}c_{2}}) \leq 2 \romd_{A}(K_{c_{1}c_{2}}) = 2 \romd_{\mathrm{sim}}(K_{c_{1}c_{2}}).
\] 
\end{theorem}

Suppose that 
\[
	\frac{\log c_{1}}{\log c_{2}} = \frac{p_{1}}{p_{2}}.
\]
Then, $c_{1}= c_{2}^{p_{1}/p_{2}}.$ Let $c = c_{2}^{1/p_{2}}.$ Then, $c_{1} = c^{p_{1}}$ and $c_{2} = c^{p_{2}}$. Moreover, $c^{p_{1}} < c^{p_{2}} < 1/4$. Thus, it suffices to prove the following result.

\begin{theorem}
Suppose $c \in (0,1)$, $p_{2}<p_{1} \in \mathbb{N}$ such that $c^{p_{1}} < c^{p_{2}} < 1/4$. Let $K=K_{c^{p_{1}}c^{p_{2}}}$ be the attractor of the system $\mathcal{F} = \{f_{1}, f_{2}\}$ where
\[
	f_{1}(x) = c^{p_{1}} x, \qquad \text{and} \qquad f_{2}(x) = c^{p_{2}} x + (1 - c^{p_{2}}).
\]
Then, $K$ satisfies the weak separation condition for differences. In particular, 
$$\romd_{A}(K-K) \leq 2 \romd_{A}(K) \leq 2 \romd_{\mathrm{sim}}(K).$$
\end{theorem}

The proof follows a similar procedure with the one for symmetric Cantor sets, but is significantly more involved.

\begin{proof}

Fix any $0< b < 1.$ Let $\aaa, \bb, \cc, \dd \in I_{b}$. Then, for any $x,y \in K$, we have
\begin{equation}\label{CALPHAS}
	f_{\aaa}(x) - f_{\bb}(y) -f_{\cc}(x) + f_{\dd}(y) = (c_{\aaa} - c_{\cc} ) x + (c_{\dd} - c_{\bb}) y + q_{\aaa} - q_{\bb} - q_{\cc} + q_{\dd},
\end{equation}
for some translations $q_{\aaa}, q_{\bb}, q_{\cc}, q_{\dd}$.
By definition of $I_{b}$, we have that for any $\alpha \in I_{b}$
\begin{equation}\label{coefficients non symmetric}
c_{\aaa} \geq c_{\bar{\aaa}} c^{p_{1}} \geq b c^{p_{1}}
\end{equation}
	
Assume that
\[
	f_{\aaa}(K) - f_{\bb}(K) \neq f_{\cc}(K) - f_{\dd}(K).
\]
First we treat the case that $q_{\aaa} - q_{\bb} - q_{\cc} + q_{\dd} = 0$. By compactness of $K$, let $x_{0}, y_{0} \in K$ such that
\[
	d_{H}(f_{\aaa}(K) - f_{\bb}(K), f_{\cc}(K) - f_{\dd}(K)) = \dist(f_{\aaa}(x_{0}) - f_{\bb}(y_{0}), f_{\cc}(K) - f_{\dd}(K)) >0.
\]
Therefore,
\[
	|f_{\aaa}(x_{0}) - f_{\aaa}(y_{0}) - f_{\cc}(x_{0}) + f_{\dd}(y_{0})| \geq \dist(f_{\aaa}(x_{0}) - f_{\bb}(y_{0}), f_{\cc}(K) - f_{\dd}(K)) >0.
\]
It is immediate from \eqref{CALPHAS} that either $c_{\aaa} \neq c_{\cc}$ or $c_{\bb} \neq c_{\dd}$.
Assume without loss of generality that $c_{\aaa} < c_{\cc}.$
Then, since $0,1 \in K$ we have by \eqref{coefficients non symmetric} 
\[
	|f_{\aaa}(1) - f_{\bb}(0) -f_{\cc}(1) + f_{\dd}(0)| = |c_{\aaa} - c_{\cc}| =  c_{\cc} \left(1 - \frac{c_{\aaa}}{c_{\cc}}\right) \geq b c^{p_{2}} \left(1 - \frac{c_{\aaa}}{c_{\cc}}\right).
\]
We claim that there exists $M >0$ such that 
\[
	\left(1 - \frac{c_{\aaa}}{c_{\cc}}\right) \geq M.
\]
Indeed, let $n_{1}, n_{2}, m_{1}, m_{2}$ be such that 
\[
	c_{\aaa} = c^{p_{1} n_{1} + p_{2} n_{2}} \qquad \text{and} \qquad c_{\cc} = c^{p_{1} m_{1} + p_{2} m_{2}},
\]
with $p_{1} m_{1} + p_{2} m_{2} < p_{1} n_{1} + p_{2} n_{2}.$
Then, 
\[
	\frac{c^{p_{1} n_{1} + p_{2} n_{2}}}{c^{p_{1} m_{1} + p_{2} m_{2}}} \leq \frac{c^{p_{1} m_{1} + p_{2} m_{2} +1}}{c^{p_{1} m_{1} + p_{2} m_{2}}} = c,
\]
which implies that
\[
	\left(1 - \frac{c_{\aaa}}{c_{\cc}}\right) \geq (1-c).
\]
Thus, the weak separation property is satisfied when $q_{\aaa} - q_{\bb} - q_{\cc} + q_{\dd} = 0.$
Suppose now that $q_{\aaa} - q_{\bb} - q_{\cc} + q_{\dd} \neq 0.$
By \eqref{CALPHAS}, we have that
\[
	|f_{\aaa}(0) - f_{\bb}(0) - f_{\cc}(0) + f_{\dd}(0)| = |q_{\aaa} - q_{\bb} - q_{\cc} + q_{\dd}|. 
\]
We claim that there exists $M_{1} >0$ such that
\[
	|q_{\aaa} - q_{\bb} - q_{\cc} + q_{\dd}| \geq M_{1} b.
\]
We want to write $q_{\aaa} - q_{\bb} - q_{\cc} + q_{\dd}$  in terms of powers of $c^{p_{1}}$ and $c^{p_{2}}$.
For $\aaa \in I_{b}$, let $n_{\aaa}, m_{\aaa} \in \mathbb{N}$ be such that
\[
	c_{\aaa} = c^{n_{\aaa} p_{1} + m_{\aaa}p_{2}}.
\]  
Then, we have 
\begin{align*}
	q_{\aaa} & = (1 - c^{p_{2}}) \left(c^{0} \left(\sum_{j=0}^{m_{\aaa}-1} t_{o \,j} c^{p_{2} \,j}\right)+ \cdots + c^{p_{1} (n_{\aaa} - 1)}\left( \sum_{j=1}^{m_{\aaa} -1 } t_{n_{\aaa} - 1 \, j} c^{p_{2} \,j}\right)\right) 
\end{align*}
for some $t_{ij} \in \{0,1\}$, where $0\leq i\leq n_{\aaa}-1$ and $0 \leq j \leq m_{\aaa}-1$.

Assume that $N_{1} = \max \{n_{\aaa}, n_{\bb}, n_{\cc}, n_{\dd}\}$ and $N_{2} = \max \{m_{\aaa}, m_{\bb}, m_{\cc}, m_{\dd}\}$.
Then, 
\begin{align*}
	q_{\aaa} - q_{\bb} - q_{\cc} + q_{\dd} =(1 - c^{p_{2}}) \left(c^{0} \left(\sum_{j=0}^{N_{2}-1} a_{o \,j} c^{p_{2} \,j}\right) + \cdots + c^{p_{1} (N_{1} - 1)}\left( \sum_{j=1}^{N_{2} -1 } a_{(N_{1} - 1) j} c^{p_{2} \,j}\right)\right) 
\end{align*}
where $a_{ij} \in \{-2,-1,0,1,2.\},$ for all $ i \leq N_{1}-1$ and $j \leq N_{2} - 1$.
Let 
\[
	A_{i} = \sum_{j=0}^{N_{2}-1} a_{ij} c^{p_{2}j}.
\]
Since $c^{p_{2}} <1/4 < 1/3$, by the argument in the previous section (see proof of Proposition \ref{symmetric cantor sets proof}), we can rewrite all negative $A_{i}$ such that all the coefficients $a_{ij}$ are negative and all positive $A_{i}$, such that all $a_{ij}$ are positive. 
In this case we also note by the previous argument that if $a_{i j} < 0$, for some $i,j$ then 
$$-2 \leq a_{i j} \leq 3 - \frac{1}{c^{p_{2}}} < -1$$ and if $a_{i j} > 0$, for some $i,j$, then 
$$1 \leq \frac{1}{c^{p_{2}}} - 3 \leq a_{i j} \leq 2.$$
Consequently, if $A_{i} < 0$ then 
\begin{equation}\label{KASA-1}
	-2  \sum_{j=1}^{\infty} c^{p_{2} j} \leq - 2 \sum_{j=1}^{N_{2}-1} c^{p_{2} j}  \leq A_{i} < - \sum_{j=1}^{N_{2}-1} c^{p_{2} j},
\end{equation}
and if $A_{i} >0$, we have
\begin{equation}\label{KASA}
	1 \leq \sum_{j=1}^{N_{2}-1} c^{p_{2} j} \leq A_{i} \leq 2 \sum_{j=1}^{N_{2}-1} c^{p_{2} j} \leq 2 \sum_{j=1}^{\infty} c^{p_{2} j}.
\end{equation}
Assume that $q_{\aaa} - q_{\bb} - q_{\cc} + q_{\dd} > 0$. We want to rewrite the above sum such that all $A_{i}$ are non--negative. If $A_{N_{1}-1} \geq 0$, we set $\widehat{A_{N_{1}-1}} = A_{N_{1}-1}$. If $A_{N_{1}-1} < 0$, we set 
\[
	\widehat{A_{N_{1}-1}} = \left(\frac{1}{c^{p_{1}}} + A_{N_{1}-1}\right) c^{p_{1}(N_{1}-1)} - c^{p_{1}(N_{1}-2)}.
\]
We claim that $\left(\frac{1}{c^{p_{1}}} + A_{N_{1}-1}\right) > 0$.
Indeed, since $A_{N_{1}-1} <0$, we have that 
\begin{equation}\label{KASA1}
	A_{N_{1}-1} \geq -2 \sum_{i=0}^{N_{2}-1} c^{p_{2} i} \geq \frac{-2}{1 - c^{p_{2}}} \geq -3,
\end{equation}
since $c^{p_{2}} < \frac{1}{3}.$ Thus, $\widehat{A_{N_{1}-1}} > 0$, since $c^{p_{1}} < \frac{1}{4}.$ Now, arguing is in the symmetric Cantor set case, if $A_{N_{1}-2} -1 \geq 0$, we set $\widehat{A_{N_{1}-2}} = A_{N_{1}-2} -1$, while if $A_{N_{1}-2} -1 < 0$, we set 
\[
	\widehat{A_{N_{1}-2}} = \left(\frac{1}{c^{p_{1}}} + A_{N_{1}-2} - 1\right),
\]
which is positive since  since $c^{p_{1}} < 1/4$ and $A_{N_{1} - 2} \geq -3$. We then write
\begin{align*}
	A_{N_{1}-2} c^{p_{1} (N_{1}-2)} & = \left(\frac{1}{c^{p_{1}}} + A_{N_{1}-2} - 1\right) c^{p_{1} (N_{1}-2)} - c^{p_{1}(N_{1}-3)}\\& =\widehat{A_{N_{1}-2}}c^{p_{1} (N_{1}-2)}   - c^{p_{1}(N_{1}-3)}.
\end{align*}
 
We continue the process until we have defined $\widehat{A_{0}}$. We note that for all $1 \leq i \leq n-1$, by \eqref{KASA} and \eqref{KASA1}, we have
\[
	\frac{1}{c^{p_{1}}} - 4 \leq \widehat{A_{i}} \leq 2 \sum_{j=1}^{N_{2}-1} c^{p_{2} j}.
\]
Hence,
\begin{align*}
	\sum_{i=1}^{N_{1}-1} \widehat{A_{i}} c^{i p_{1}} \leq  2 B_{j} \sum_{i=1}^{\infty} c^{p_{1} i} ,
\end{align*}
where 
$$B_{j} = \sum_{j=1}^{N_{2}-1} c^{p_{2} j}.$$
Therefore,
\[
	\sum_{i=1}^{N_{1}-1} \widehat{A_{i}} c^{i p_{1}} \leq B_{j} \left(\frac{2 c^{p_{2}}}{1 - c^{p_{2}}}\right) < B_{j},
\]
since $c^{p_{2}} < 1/3$.
Since we have assumed $q_{\aaa} - q_{\bb} - q_{\cc} +q_{\dd} >0$, we need 
\[
	\widehat{A_{0}} > - B_{j} = - \sum_{j=1}^{N_{2}-1} c^{p_{2} j}.
\]
But, from \eqref{KASA-1}, \eqref{KASA}, we deduce that if $\widehat{A_{0}} <0$, then it must satisfy 
\[
		\widehat{A_{0}} \leq - \sum_{j=1}^{N_{2}-1} c^{p_{2} j}.
\]
Consequently, $\widehat{A_{0}} \geq 0$. By a symmetric argument, if $q_{\aaa} - q_{\bb} - q_{\cc} +q_{\dd} <0$, we can rewrite the sum such that all the $A_{i}$ are non--positive, for $0 \leq i \leq N-1$. Assume without loss of generality that $q_{\aaa} - q_{\bb} - q_{\cc} +q_{\dd} > 0$. Then, we have 
\[
	 q_{\aaa} - q_{\bb} - q_{\cc} +q_{\dd} = (1 - c^{p_{2}})\sum_{i=0}^{N_{1}-1} \widehat{A_{i}} c^{i p_{1}}.
\]
By the above construction, we deduce that if $\widehat{A_{i}} > 0$, then $\widehat{A_{i}} \geq \frac{1}{c^{p_{1}}} - 4  > 0,$ for $0 \leq i \leq N_{1}-1$. Moreover, for every $i$, we have that
\[
	\widehat{A_{i}} = \sum_{j=0}^{N_{2}-1} \widehat{a_{ij}} c^{j p_{2}}.
\]
Similarly, if  $\widehat{a_{ij}} >0$, for some $0 \leq j \leq N_{2} - 1$, then $\widehat{a_{ij}} \geq \frac{1}{c^{p_{2}}} - 3 >0$.

Assume that $N_{1} = n_{\aaa} = \max\{n_{\aaa}, n_{\bb}, n_{\cc}, n_{\dd}\}$ and $N_{2} = m_{\bb}.$ Assume without loss of generality that $m_{\bb} \leq n_{\aaa}$. Thus, $m_{\aaa} \leq m_{\bb} \leq n_{\aaa}$. Let $0 \leq m \leq n_{\aaa} - 1$ such that $\widehat{A_{m}} >0$. Let also $0 \leq n \leq m_{\aaa} -1 \leq m_{\bb} -1$ such that $\widehat{a_{mn}} >0$.
Thus,
\begin{align*}
	|q_{\aaa} - q_{\bb} - q_{\cc} + q_{\dd}| & \geq (1 - c^{p_{2}}) \widehat{A_{m}}c^{p_{1}m} \geq (1 - c^{p_{2}}) \widehat{a_{m n}} c^{p_{2} n} c^{p_{1} m}\\
	&  \geq (1 - c^{p_{2}}) \left(\frac{1}{c^{p_{2}}} - 3\right) c^{(m_{\aaa}-1) p_{2}} c^{(n_{\aaa}-1) p_{1}}\\
	& \geq (1 - c^{p_{2}}) \left(\frac{1}{c^{p_{2}}} - 3\right) c_{\bar{\aaa}} \geq (1 - c^{p_{2}}) \left(\frac{1}{c^{p_{2}}} - 3\right) b,
\end{align*}
which concludes the proof that the system satisfies the weak separation condition for differences. Since the weak separation condition for differences trivially implies the standard weak separation condition we deduce that $\romd_{\mathrm{sim}} (K) \geq \romd_{A}(K)$, which implies the desired result by Theorem \ref{Differences}.
\end{proof}

The above theorem can be a useful tool for computing explicit bounds for the Assouad dimension of differences of asymmetric Cantor sets. Let $c \in (0,1)$, $c_{1} = c^{p}, c_{2} = c^{2p}$ such that $c^{2p} < c^{p} < \frac{1}{4}.$ Then, we can explicitly compute the similarity dimension $\romd_{\mathrm{sim}}(A_{c_{1}c_{2}}).$ In particular, let $D$ such that 
\[
	c^{p D} + c^{2 p D} = 1.
\]
By solving the quadratic equation for $c^{p}$, we find that 
\[
	D = \frac{\log \phi}{p \log \left(\frac{1}{c}\right)},
\]
where 
\[
	\phi= \frac{2}{\sqrt{5} -1}.
\]
Thus,
\[
	\romd_{A}(A_{c_{1}c_{2}} - A_{c_{1}c_{2}}) \leq \frac{2 \log \phi}{p \log \left(\frac{1}{c}\right)}.
\]

We note that in the above argument we only require one of the exponents to be less than $1/4$ and the other one to be less than $1/3$. However, for simplicity, we assume that both are less than $1/4$.

\section{Conclusion}
We showed that when a system of contracting similarities satisfies a suitable separation condition, then the attractor of the system possesses a set of differences that obeys non-trivial bounds related to the Assouad dimension of the attractor itself. In particular, we show that particular examples of symmetric and asymmetric Cantor sets fall in the above class. There are a number of questions that arise naturally from these results and we would like to list some of them.

\begin{Q}
In Lemma \ref{sufficient condition for weak separation}, we give a sufficient condition for the weak separation to hold. Is it true that the two conditions are equivalent?
\end{Q}

\begin{Q}
Henderson \cite{Henderson} showed that if $\frac{\log c_{1}}{\log c_{2}}$ is irrational then the Assouad dimension of the set of differences is maximal. Is it true that when $c_{1} < 1/4$, $c_{2} \geq 1/4$ and  $\frac{\log c_{1}}{\log c_{2}}$ is rational, then the weak separation property for differences is always satisfied?
\end{Q}

\begin{Q}
Suppose $H$ is a Hilbert space.
Let $f \colon H \to H$ be a contracting similarity, i.e. it satisfies  
\[
	\|f(x) - f(y)\| = c \|x-y\|,
\]
for all $x,y \in H$ and for some $c <1$.
We know by Hutchinson \cite{HUTCH} that there exists a unitary operator $U\colon H \to H$ and a point $q \in H$ such that 
\[
	f(x) = c \, U(x) + q.
	\]
In particular, $f$ is bijective and the inverse $f^{-1}$ is a similarity that satisfies
\[
	\|f^{-1}(x) - f^{-1}(y)\| = \frac{1}{c} \|x-y\|,
\]
for all $x,y \in H$. Suppose $\mathcal{F} = \{f_{i} \colon H \to H\}$ is a system of similarities like the one described above with an attractor $K$. What we can we say about the Assouad dimension of $K$ in this case? Can we formulate a separation condition, similar to the one we introduced in this paper  and show that $d_{A}(K-K)$ is finite under that condition?
\end{Q}

\bibliographystyle{plain}
\bibliography{References}
\end{document}